\documentclass[reqno]{amsart}

\usepackage{amssymb}

\def\Conf{\mathop{\fam 0 Conf}\nolimits}
\def\Cur{\mathop{\fam 0 Cur}\nolimits}
\def\Vir{\mathop{\fam 0 Vir}\nolimits}
\def\oo#1{\mathrel {{}_{(#1)}}}

\newtheorem{thm}{Theorem}
\newtheorem{lem}{Lemma}
\newtheorem{prop}{Proposition}
\newtheorem{cor}{Corollary}
\newtheorem{remark}{Remark}

\theoremstyle{definition}

\newtheorem{definition}{Definition}
\newtheorem{example}{Example}

\title[Gr\"obner---Shirshov Bases for Conformal Algebras]{Gr\"obner---Shirshov 
Bases for Associative\\ Conformal Algebras With Arbitrary Locality Function}

\author{P. Kolesnikov}

\address{Sobolev Institute of Mathematics,\\ Akad. Koptyug prosp., 4\\ 630090 Novosibirsk, Russia\\
Email: pavelsk@math.nsc.ru}



\begin{document}

\begin{abstract}
 We present an approach to the computation of confluent systems of defining relations 
 in associative conformal algebras based on the similar technique for 
 modules over ordinary associative algebras. 
\end{abstract}

\maketitle

\begin{center}
 Dedicated to my Teacher Leonid A. Bokut to his 80th birthday
\end{center}

\section{Introduction}

Gr\"obner---Shirshov bases (GSB) is a useful tool for solving theoretical problems in algebra (see, e.g., \cite{BokChen13}. 
It is not so common in applications as Gr\"obner bases in commutative algebra 
since the construction of a confluent rewriting system for a given ideal may not 
be done in a finite number of steps.  

In order to establish a version of GSB theory for a particular class of algebraic systems
we usually need to determine the structure of free objects (determine what is a normal word),
define a linear order on the set of normal words compatible with algebraic structure (i.e., 
monomial order), define a sort of elimination procedure, and state an appropriate version 
of the Composition-Diamond Lemma (CD-Lemma). The latter is the key statement of the theory, a criterion 
for a set of defining relations to be confluent. 

However, there are many cases when realization of the entire program described above is excessive. 
For example, GSB theory for modules over associative algebras \cite{KangLee} can be 
explicitly deduced from the classical CD-Lemma for associative algebras (see Theorem~\ref{thm:CD-mod} below), 
GSB theory for di-algebras \cite{BCL10, ZhangChen17} may also be considered as an application of the same 
statement \cite{Kol2017AC}. In a similar way, GSB theory for Leibniz algebras may be derived 
from the Shirshov's CD-Lemma for Lie algebras.

Conformal algebras appeared in mathematical physics as ``singular parts'' of vertex algebras \cite{Kac98}.
Combinatorial study of associative conformal algebras was started in \cite{Ro99}. 
In a series of papers \cite{BFK00, BFK2004, NiChen2017}, different versions of the CD-Lemma 
for associative conformal algebras was stated. 
In this paper, we show how to derive GSB theory for associative conformal algebras from 
CD-Lemma for modules over (ordinary) associative algebras. This approach is technically simpler
than and it is more general: we do not assume the locality of generators is bounded.
Moreover, in this way one may use existing computer algebra 
systems to calculate compositions in conformal algebras. 
As an application, we calculate GSB for universal associative envelopes of 
some Lie conformal algebras.

Throughout the paper, $\Bbbk $ is a base field of 
characteristic zero, $\mathbb Z_+$ is the set of 
non-negative integers. Given a set $X$, $X^*$ stands 
for the set of nonempty words in the alphabet $X$, 
$X^\#$ denotes the set $X^*$ together with the
empty word. We will use notation like $x^{(s)}$ 
for $\frac{1}{s!} x^s$ for $s\in \mathbb Z_+$.

\section{Preliminaries: CD-Lemma for modules}

Let us recall the classical CD-lemma for associative algebras in the form of \cite{Bok76}. 

Suppose $O$ is a set of generators and $\preceq $ is a monomial order on $O^*$. 
As usual, $\bar f$ denotes the principal monomial of a nonzero polynomial
from the free associative algebra $\Bbbk \langle O\rangle$ relative to the order $\preceq $.

If $f$ and $g$ are monic polynomials in $\Bbbk \langle O\rangle$
such that $w=\bar f = u\bar g v$ for some $u,v\in O^\#$, then 
\[
 (f,g)_w = f-ugv
\]
is said to be a composition of inclusion.

If $\bar f = u_1u_2$, $\bar g = u_2v_2$ for some $u_1,u_2,v_2\in O^*$ then 
\[
 (f,g)_w = fv_2-u_1g, \quad w=u_1u_2v_2,
\]
is said to be a composition of intersection.

Given a polynomial $f\in \Bbbk\langle O\rangle$, a set of polynomials $\Sigma \subset \Bbbk \langle O\rangle$, 
and a word $w\in O^*$, we say $f$ is trivial modulo $\Sigma $ and $w$ if $f$ can be presented in the form
\[
 f = \sum\limits_i \alpha_i u_is_iv_i,\quad \alpha_i\in \Bbbk ,\ u_i,v_i\in O^\#, \ s_i\in \Sigma ,
\]
where $u_i\bar s_iv_i \prec w$.
We will use the following notation:
\[
 f\equiv g \pmod {\Sigma ,w}
\]
if $f-g$ is trivial modulo $\Sigma $ and $w$.

A word $u\in O^*$ is said to be $\Sigma $-reduced if 
$u\ne v_1\bar s v_2$, $v_1,v_2\in O^\#$, $s\in \Sigma $.

A set of monic polynomials $\Sigma \subset \Bbbk \langle O\rangle $ 
is said to be a Gr\"obner---Shirshov basis  (GSB)
if $(f,g)_w\equiv 0 \pmod {\Sigma ,w}$ for every pair of polynomials  $f,g\in \Sigma $.

By abuse of terminology, we say $\Sigma $ is a GSB of the ideal $I=(\Sigma )$ in $\Bbbk \langle O\rangle$.

Note that a pair of polynomials in $\Bbbk \langle O\rangle$ may have more than one composition. 
In order to make sure that $\Sigma $ is a GSB we have to check if all compositions are trivial 
in the above-mentioned sense.

The following statement is known as 
the CD-Lemma for associative algebras~\cite{Bok76}.

\begin{thm}\label{thm:CD-ass}
 For a set $\Sigma $ of monic polynomials in $\Bbbk \langle O\rangle$ the following statements are equivalent:
 \begin{enumerate}
  \item $\Sigma $ is a GSB;
  \item If $f\in (\Sigma )$ then $\bar f = u\bar s v$ for appropriate $u,v\in O^\#$, $s\in \Sigma $;
  \item The set of $\Sigma $-reduced words forms a linear basis of $\Bbbk \langle O\mid \Sigma \rangle $.
 \end{enumerate}
\end{thm}

The same approach works for modules over associative algebras. 
Suppose $O$ is a set of generators of an associative algebra $A$, 
$X$ is a set of generators of a left $A$-module $M$.
Let $\Sigma \subset \Bbbk \langle O\rangle $ be a set of defining 
relations of $A$ and let $S\subset \Bbbk \langle O\rangle \otimes \Bbbk X$ 
be a set of defining relations of $M$.
We will identify 
$\Bbbk \langle O\rangle \otimes \Bbbk X$ with 
a subset 
$\Bbbk \langle O\dot\cup X \rangle_1 $
of
$\Bbbk \langle O\dot\cup X \rangle$ 
that consist of polynomials 
\[
 \sum\limits_i \alpha _iu_ix_i,\quad \alpha_i\in \Bbbk,\ u_i\in O^\#,\ x_i\in X.
\]

Consider the alphabet $Z=O\dot\cup X$ and a monomial order on $Z^*$. 
Obviously, $M$ is a subspace of the associative algebra $A(M)$
generated by $Z$ relative to the defining relations $\Sigma \cup S$ 
along with 
$xa$, $xy$ $x,y\in X$, $a\in O$. Indeed, $A(M)$ is the split null extension 
of $A$ by means of $M$.

\begin{definition}
Let $\Sigma $ be a GSB in $\Bbbk \langle O\rangle $, $A=\Bbbk \langle O\mid \Sigma \rangle $, 
$F(X)$ is the free $A$-module generated by $X$.
A set $S$ of monic polynomials in $\Bbbk \langle O\dot\cup X \rangle _1$ 
is said to be a Gr\"obner---Shirshov basis in $F(X)$ 
if $S\cup\Sigma \cup \{xa, xy\mid x,y\in X,\ a\in O\}$ is a GSB in $\Bbbk \langle O\dot\cup X \rangle $.
\end{definition}

The following statement is equivalent 
to the CD-lemma for modules proved in~\cite{KangLee}.

\begin{thm}\label{thm:CD-mod}
 For a set of monic polynomials $S\subset \Bbbk \langle O\dot\cup X \rangle_1 $ 
 the following statements are equivalent:
\begin{enumerate}
 \item $S$ is a GSB in $F(X)$;
 \item If $f\in \Bbbk \langle O\dot\cup X \rangle_1$ belongs to 
 the $A$-submodule generated by $S$ then either 
 $\bar f = u\bar svx$, $u,v\in O^\#$, $s\in \Sigma $, $x\in X$, 
 or $\bar f = u\bar g$, $u\in O^\#$, $g\in S$;
 \item 
 The set of $(\Sigma \cup S)$-reduced set of words in $\Bbbk \langle O\dot\cup X \rangle_1$
 forms a linear basis of the $A$-module $M$ generated by $X$ relative to the defining relations $S$.
\end{enumerate}
\end{thm}

\section{Conformal algebras}

Conformal algebra \cite{Kac98} is a linear space $C$ equipped with a linear map $\partial : C\to C$ 
and a family of bilinear products $(\cdot \oo{n}\cdot )$, $n\in \mathbb Z_+$, such that 
the following axioms hold: for every $a,b\in C$
\begin{gather}
 a\oo{n} b = 0 \quad \text{for almost all $n\in \mathbb Z_+$};
                       \label{eq:Locality}           \\
 \partial a \oo{n} b = -na\oo{n-1} b,\quad a\oo{n}\partial b = \partial(a\oo{n} b) + na\oo{n-1} b.
                       \label{eq:3/2-linearity}
\end{gather}
Locality function on $C$ is a map 
$N_C: C\times C \to \mathbb Z_+$, where $N_C(a,b)$ is equal to the minimal $N$ such that 
$a\oo{n} b=0$ for all $n\ge N$.

If 
\[
 a\oo{n} (b\oo{m} c ) = \sum\limits_{s\ge 0} (-1)^s \binom{n}{s} (a\oo{n-s} b)\oo{m+s} c 
\]
for all $a,b,c\in C$ and $n,m\in \mathbb Z_+$ then 
$C$ is said to be associative. 

Axiom \eqref{eq:Locality} shows the class of associative conformal algebras is not a variety 
in the sense of universal algebra. However, for a given set $X$ of generators and for a fixed 
function $N:X\times X \to \mathbb Z_+$ there exists 
a free object in the class of associative conformal algebras $C$ generated by $X$ such that 
$N_C(a,b)\le N(a,b)$ for $a,b\in X$ \cite{Ro99}. This free algebra is unique up to isomorphism, 
let us denote it $\Conf(X,N)$. As shown in \cite{Ro99}, $\Conf(X,N)$ has a 
linear basis which consists of words
\[
\begin{gathered}
 \partial^{s} (a_1\oo{n_1} (a_2 \oo{n_2}  \dots (a_k \oo{n_k} a_{k+1})\dots )), \\
 s\ge 0,\ k\ge 1, \ a_i\in X,\ 0\le n_i<N(a_{i},a_{i+1}).
\end{gathered}
\]

Following \cite{Kac99}, introduce the family of operations 
\[
 \{a\oo{n} b\} = \sum\limits_{s\ge 0} (-1)^{n+s} \partial^{(s)} (a\oo{n+s} b),\quad a,b\in C,\ n\in \mathbb Z_+.
\]
As shown in \cite{Kac99}, in an associative conformal algebra we have 
\[
\begin{gathered}
 \{\partial a\oo{n} b\} = \partial\{a\oo{n} b\} + n\{a\oo{n-1} b\}, \quad \{a\oo{n} \partial b\} = -n\{a\oo{n-1} b\},\\
 a\oo{n}\{b\oo{m} c\} = \{(a\oo{n} b)\oo{m} c\}.
\end{gathered}
\]

\section{Module construction of the free associative conformal algebra}\label{eq:Sect4}

Let $(X, \le )$ be a well-ordered set, and $N: X\times X\to \mathbb Z_+$ be an arbitrary function.
Consider the set of symbols
\[
 O=\{\partial, L_n^a, R_n^a \mid n\in \mathbb Z_+,\ a\in X \}
\]
Define a monomial order $\preceq $ on $O^*$ in the following way:
compare two words first by their degree in 
the variables $R_n^a$, then by deg-lex order assuming
$\partial <L_n^a<R_m^b$, 
$L_n^a <L_m^b$ if $n<m$ or $n=m$ and $a<b$; similarly for $R_n^a$, $R_m^b$.

Denote by $A(X)$ the associative algebra generated by $O$ relative to the following 
defining relations:
 \begin{gather}
 L_n^a\partial - \partial L_n^a - nL_{n-1}^a,  \label{eq:Alg_A:Relations1} \\
 R_n^a\partial - \partial R_n^a - nR_{n-1}^a,  \label{eq:Alg_A:Relations2} \\
 R_m^bL_n^a - L_n^aR_m^b,   \label{eq:Alg_A:Relations3}
\end{gather}
$a,b\in X$, $n,m \in \mathbb Z_+$.

\begin{lem}\label{lem:BGS_A}
 Polynomials \eqref{eq:Alg_A:Relations1}--\eqref{eq:Alg_A:Relations3} form a GSB
 with respect to the order $\preceq $ on~$O^*$.
\end{lem}

\begin{proof}
 The only composition here is the composition of intersection of \eqref{eq:Alg_A:Relations1}
 and \eqref{eq:Alg_A:Relations3}. It is straightforward to check that this composition is trivial.
\end{proof}

Denote by $F(X)$ the free (left) $A(X)$-module generated by $X$. Lemma~\ref{lem:BGS_A} 
implies that the linear base of $F(X)$ consists of words 
\begin{equation}\label{eq:Mod_F:basis}
 \partial^s L_{n_1}^{a_1}\dots L_{n_k}^{a_k} R_{m_1}^{b_1}\dots R_{m_p}^{b_p}c,
 \quad 
 s,k,p,n_i,m_j\in \mathbb Z_+, \ a_i,b_j,c\in X.
\end{equation}

\begin{definition}\label{defn:NormalWord}
Let us call a word of type \eqref{eq:Mod_F:basis} {\em normal} if 
$p=0$, $n_i<N(a_i,a_{i+1})$ for $i=1,\dots, k-1$, 
and
$n_k <N(a_k,c)$.
Denote by $B(X,N)$ the set of all normal words, and use
 $B_0(X,N)$ for the set of $\partial $-free 
normal words.
\end{definition}

Let $M(X,N)$ stand for the $A(X)$-module generated by $X$ relative to the following 
relations:
\begin{gather}
 L_n^a b, \quad a,b\in X,\ n\ge N(a,b),                       \label{eq:Mod_M:Locality}\\
 R_n^b a - (-1)^n\sum\limits_{s=0}^{N(a,b)-n-1} \partial^{(s)} L_{n+s}^a b,
    \quad a,b\in X,\ n\in \mathbb Z_+.                        \label{eq:Mod_M:RightMul}
\end{gather}
This module is the main object of study in this section.
The following statement explains the origin of $A(X)$ and $M(X,N)$.

\begin{prop}\label{prop:FreeConfMod}
 The free associative conformal algebra $\Conf(X,N)$ 
 is an $A(X)$-module which is a homomorphic image of $M(X,N)$.
\end{prop}

\begin{proof}
 The action of $A(X)$ on $\Conf(X,N)$ is given by 
 \[
  L_n^a u = a\oo{n} u,\quad R_n^a u = \{u\oo{n} a\}
 \]
for $a\in X$, $u\in \Conf(X,N)$, $n\in \mathbb Z_+$,  $\partial $ acts naturally. 
It is easy to see that these rules are compatible with the defining relations of $A(X)$,
as well as \eqref{eq:Mod_M:Locality}, \eqref{eq:Mod_M:RightMul} hold. 
Obviously, $\Conf(X,N)$ as an $A(X)$-module is generated by~$X$. 
\end{proof}

Let us extend the order $\preceq $ to the set of words $O^*\cup O^\#X$ which occur in the computation
of compositions in the free $A(X)$-module generated by $X$.
For $ux,vy \in O^\#X$, where $u,v\in O^\#$, $a,y\in X$, assume $ux\prec vy$ if and only if $u\prec v$ or $u=v$ and $x<y$.
Moreover, set $O^*\prec O^\#X$. 

\begin{thm}\label{thm:Mod_M:BGS}
 Gr\"obner---Shirshov basis of the $A(X)$-module $M(X,N)$
 consists of 
 \eqref{eq:Mod_M:Locality}, \eqref{eq:Mod_M:RightMul} and
\begin{equation}\label{eq:Mod_M:Locality_ex}
L_n^aL_m^bu +\sum\limits_{q\ge 1} (-1)^q \binom{n}{q} L_{n-q}^aL_{m+q}^b u,
\end{equation}
where $a,b\in X$, $n\ge N(a,b)$, $m\in \mathbb Z_+$, $u\in B_0(X,N)$. 
\end{thm}

\begin{proof}
Let us show that relations \eqref{eq:Mod_M:Locality_ex} 
follow from the defining relations 
\eqref{eq:Mod_M:Locality}, \eqref{eq:Mod_M:RightMul}.
For $u=c\in X$, consider the compositions of intersection of
\[
 f=R_m^cL_n^a - L_n^aR_m^c,\quad g= L_n^ab
\]
for $n\ge N(a,b)$, $m\in \mathbb Z_+$:
\begin{multline}\nonumber
(f,g)_{R_m^cL_n^ab} =  (R_m^cL_n^a - L_n^aR_m^c)b - R_m^c(L_n^ab) \\
 \equiv (-1)^{m+1} \sum\limits_{s=0}^{N(b,c)-1} (-1)^s L_n^a\partial^{(s)}L_{m+s}^b c \\
 \equiv (-1)^{m+1} \sum\limits_{s=0}^{N(b,c)-1}
   \sum\limits_{t\ge 0}
   (-1)^s \binom{n}{t} \partial^{(s-t)} L_{n-t}^a L_{m+s}^b c.
\end{multline}
These compositions are trivial for $m\ge N(b,c)$, for $m=N(b,c)-1$ we obtain 
$(f,g)_{R_m^cL_n^ab}$ 
to be a multiple of
$L_{n}^a L_{m}^b c $. 
For smaller $m$, proceed by induction. Suppose 
\eqref{eq:Mod_M:Locality_ex} holds for $u=c$ and all $m>p$. 
Then 
\begin{multline}\label{eq:CompInter}
 (f,g)_{R_p^cL_n^ab}
\equiv 
 (-1)^{p+1} \sum\limits_{s=0}^{N(b,c)-1}
   \sum\limits_{t\ge 0}  (-1)^s  \binom{n}{t} \partial^{(s-t)} L_{n-t}^a L_{p+s}^b c  \\
=
  (-1)^{p+1} \bigg( L_{n}^a L_{p}^b c
   +
  \sum\limits_{s=1}^{N(b,c)-1}
  (-1)^s\bigg( \partial^{(s)} L_{n}^a L_{p+s}^b c +
   \sum\limits_{t\ge 1}   \binom{n}{t} \partial^{(s-t)} L_{n-t}^a L_{p+s}^b c
   \bigg )
   \bigg )  \\
\equiv 
  (-1)^{p+1} \bigg( L_{n}^a L_{p}^b c
   +
  \sum\limits_{s=1}^{N(b,c)-1}
  \bigg( \partial^{(s)} \sum\limits_{q\ge 1}(-1)^{q+s+1} \binom{n}{q} L_{n-q}^a L_{p+s+q}^b c \\
+
   \sum\limits_{t\ge 1}   (-1)^s\binom{n}{t} \partial^{(s-t)} L_{n-t}^a L_{p+s}^b c
   \bigg )
   \bigg )   \\
= 
  (-1)^{p+1} \bigg( L_{n}^a L_{p}^b c
   +
  \sum\limits_{s,q\ge 1}
   (-1)^{q+s+1} \binom{n}{q} \partial^{(s)}  L_{n-q}^a L_{p+s+q}^b c \\
+
   \sum\limits_{r\ge 0,t\ge 1}   (-1)^{r+t}\binom{n}{t} \partial^{(r)} L_{n-t}^a L_{p+r+t}^b c
   \bigg )   \\
=
 (-1)^{p+1} \bigg( L_{n}^a L_{p}^b c
   +
  \sum\limits_{q\ge 1}
   (-1)^{q} \binom{n}{q} L_{n-q}^a L_{p+q}^b c
   \bigg ).
\end{multline}
Here we have applied substitution $r=s-t$ for the last sum in the right-hand side of \eqref{eq:CompInter}. 
Therefore, relation \eqref{eq:Mod_M:Locality_ex} holds for $m=p$ as well.

Denote by $\Sigma(X,N)$ the set of relations 
\eqref{eq:Mod_M:Locality}, \eqref{eq:Mod_M:RightMul}, and \eqref{eq:Mod_M:Locality_ex}.
Note that the set of normal words is exactly the set of $\Sigma(X,N)$-reduced words.
Images of these words under the homomorphism from Proposition~\ref{prop:FreeConfMod} 
are linearly independent in 
$\Conf(X,N)$ as shown in \cite{Ro99}.
Hence, $B(X,N) $ is a linear basis of $M(X,N)$ and thus $\Sigma(X,N)$ 
is a GSB. 
\end{proof}

\begin{remark}
 It is not hard to check in a straightforward way that 
 $\Sigma(X,N)$ is closed with respect to compositions.
\end{remark}

\begin{cor}
 Free associative conformal algebra $\Conf(X,N)$ as $A(X)$-module is isomorphic to $M(X,N)$. 
 The set of normal words $B(X,N)$ is a linear basis of the module $M(X,N)$.
\end{cor}

\begin{cor}
 Ideals of $\Conf(X,N)$ are exactly $A(X)$-submodules of $M(X,N)$. 
\end{cor}

\begin{proof}
Every (two-sided) ideal of $\Conf(X,N)$ is obviously closed with respect to the 
action of $\partial$, $L_n^a$, and $R_n^a$. Conversely, every $A(X)$-submodule 
of $M(X,N)$ is $\partial$-invariant and closed relative to left and right 
conformal multiplications on $a\in X$.
\end{proof}

Let us summarize the results above to state the CD-Lemma for associative 
conformal algebras.

\begin{definition}
 Suppose $S$ is a set of elements in $\Conf(X,N)$. Identify 
 $S$ with a set of linear combinations of normal words $B(X,N)$ in $F(X)$. 
 If $S$ together with $\Sigma(X,N)$
 is a Gr\"obner---Shirshov basis (GSB) in the $A(X)$-module $F(X)$ then we say that $S$ is a GSB in $\Conf(X,N)$.
\end{definition}

By abuse of terminology, we say a GSB $S$ in $\Conf(X,N)$ is a GSB of ideal generated by $S$ 
in $\Conf(X,N)$.

\begin{thm}\label{thm:CDLemma}
The following statements are equivalent:
\begin{itemize}
 \item $S$ is a GSB in $\Conf(X,N)$;
 \item The set of $S$-reduced normal words forms a linear basis of $\Conf(X,N\mid S)$. 
\end{itemize}
\end{thm}

\section{Applications}

One of the most interesting questions in the theory of conformal algebras 
is to determine if every torsion-free finite Lie conformal (super)algebra $L$
with operation $[\cdot \oo{\lambda } \cdot ]$
is special, i.e., whether it can be embedded into an associative 
conformal algebra $C$ in such a way that 
\begin{equation}\label{eq:ConfComm}
 [x\oo{\lambda } y] = (x\oo{\lambda } y) -\{y\oo{\lambda } x \}, \quad x,y\in L.
\end{equation}
One of the most natural ways to resolve the speciality problem is to apply GSB theory. 
Namely, suppose $X$ is a basis of $L$ as of $\Bbbk [\partial ]$-module. The multiplication 
table in $L$ is given by 
\begin{equation}\label{eq:LieMultTable}
 [a\oo{\lambda } b ] = \sum\limits_{c\in X} f_c^{a,b}(\partial,\lambda ) c,\quad f_{c}^{a,b}(\partial,\lambda)\in \Bbbk[\partial ,\lambda],
 \ a,b\in X.
\end{equation}
Obviously, $L$ is special if and only if there exists a function $N:X\times X\to \mathbb Z_+$ 
such that the GSB of the ideal generated by the set 
\begin{equation}\label{eq:CommMultTable}
a\oo{n} b - \{b\oo{n} a\} - g_n^{a,b} , \quad a,b\in X,\ n\in \mathbb Z_+, 
\end{equation}
where $g_n^{a,b}=[a\oo{n} b]$ is the coefficient of the right-hand side of \eqref{eq:LieMultTable}
at $\lambda ^{(n)}$, 
does not contain nonzero $\Bbbk[\partial]$-linear combinations of elements from $X$.
The corresponding quotient conformal algebra is the universal associative envelope 
of $L$ relative to the locality function $N$ on generators $X$ \cite{Ro2000}. 
Let us denote it by $U(L;X,N)$.

As an $A(X)$-module, $U(L;X,N)$ is generated by $X$ relative to the defining relations 
\eqref{eq:Mod_M:Locality},
\eqref{eq:Mod_M:RightMul},
and \eqref{eq:CommMultTable}.
Sesqui-linearity of the conformal $\lambda $-product implies that if $[x\oo{\lambda } y]$
is given by \eqref{eq:ConfComm} in an associative conformal algebra $C$ then 
\[
 [x\oo{\lambda } y]\oo{\lambda+\mu} z = x\oo{\lambda }(y\oo{\mu } z) - y\oo{\mu}(x\oo{\lambda }z )
\]
for $x,y,z\in C$. 
Therefore, the following relations hold on $U(L;X,N)$:
\[
 L^a_nL^b_m u - L^b_mL_n^a u - \sum\limits_{s\ge 0}\binom{n}{s} L^{g_s^{a,b}}_{n+m-s} u,\quad a,b\in X,\ n,m\in \mathbb Z_+
\]
for every normal word $u\in B(X,N)$ (in the last summand,
we assume $L^{\partial x}_n = -nL_{n-1}^x$).
Therefore, in order to analyze the structure of $U(L;X,N)$ we may replace 
$A(X)$ with the algebra $A(L;X)$ generated by $X$ relative to the defining 
relations 
\eqref{eq:Alg_A:Relations1}--\eqref{eq:Alg_A:Relations3} along with 
\begin{equation}\label{eq:CommL-operators}
 L^a_nL^b_m  - L^b_mL_n^a  - \sum\limits_{s\ge 0}\binom{n}{s} L^{g_s^{a,b}}_{n+m-s} ,\quad a,b\in X,\ n,m\in \mathbb Z_+,\ L_n^a>L_m^b.
\end{equation}

\begin{prop}\label{prop:U(L;X,N)-algebra}
 Relations 
 \eqref{eq:Alg_A:Relations1}--\eqref{eq:Alg_A:Relations3} and 
 \eqref{eq:CommL-operators}
 form a GSB in the free associative algebra generated by $O$
 with respect to the order $\preceq $.
\end{prop}

\begin{proof}
A relation of type \eqref{eq:CommL-operators} has compositions of intersection with 
\eqref{eq:Alg_A:Relations1} and \eqref{eq:Alg_A:Relations3}. It is straightforward to 
check these compositions are trivial. 
Two relations of type \eqref{eq:CommL-operators} may also have a composition of intersection. 
Such a composition is also trivial since \eqref{eq:CommL-operators} is nothing but 
the multiplication table of a Lie algebra spanned by operators $L_n^a$, $n\in \mathbb Z_+$, $a\in X$.
To be more precise, 
$A(L;X)$ 
is the universal associative envelope of the Lie algebra
$\Bbbk \partial \ltimes (\mathcal A(L)\ast \mathcal R(X)/([\mathcal A(L),\mathcal R(X)]))$, 
where $\mathcal A(L)$ is the annihilation algebra of $L$ (positive part 
of the coefficient algebra), $\mathcal R(X)$ is the free Lie algebra generated 
by $R_n^a$, $\ast $ denotes free product 
of Lie algebras, 
and $\partial $ acts as described by \eqref{eq:Alg_A:Relations1}, \eqref{eq:Alg_A:Relations2}.

The classical Poincar\'e---Birkhoff---Witt Theorem states that 
the multiplication table of a Lie algebra is a GSB relative to the deg-lex order based on an 
arbitrary ordering of the generators. The order $\preceq$ we use is not exactly deg-lex, 
but all relations \eqref{eq:Alg_A:Relations1}--\eqref{eq:Alg_A:Relations3},  \eqref{eq:CommL-operators}
are homogeneous relative to the variables $R_n^a$, so we may still conclude 
that this is a GSB.
\end{proof}

\begin{remark}
The last statement in the proof of Proposition \ref{prop:U(L;X,N)-algebra}
is equivalent to so called $1/2$-PBW Theorem for conformal algebras \cite{BFK2004}.
\end{remark}

Therefore, in order to determine if $L$ is embedded into its universal envelope $U(L;X,N)$
we have to find GSB of the $A(L;X)$-module generated by $X$ relative to the defining relations 
\eqref{eq:Mod_M:RightMul},
\eqref{eq:Mod_M:Locality},
and 
\begin{equation}\label{eq:Mod_M:Comm}
 L_n^ab - R_n^ab - g_n^{a,b},\quad a,b\in X,\ n\in \mathbb Z_+. 
\end{equation}
Let us consider two examples of such a computation.

\begin{example}
Lie conformal algebra $L=\Vir\ltimes \Cur\Bbbk $ is generated as a $\Bbbk[\partial ]$-module 
by $X=\{v,h\}$, where 
\[
 \begin{gathered}
{}
 [v\oo{0} v] = \partial v, \quad [v\oo{1} v] = 2 v, \\
 [v\oo{0} h] = \partial h,\quad
 [v\oo{1} h] = h, \quad  [h\oo{1} v]= h, 
\end{gathered}
\]
other products are zero.
\end{example}

For the Heisenberg---Virasoro Lie conformal algebra $L$, 
$A(L;X)$ is generated by $O=\{\partial, L_n^h, R_n^h, L_n^v, R_n^v \mid n\in \mathbb Z_+\}$ 
relative to the following defining relations:
\begin{gather}
{}
L_n^a\partial -\partial L_n^a - nL_{n-1}^a, \quad n\ge 0,\ a\in X,     \label{eq:HV-Ld}\\
R_n^a\partial -\partial R_n^a - nR_{n-1}^a, \quad n\ge 0,\ a\in X,     \label{eq:HV-Rd}\\
R_n^b L_m^b - L_m^bR_n^a,\quad n,m\ge 0,\ a,b\in X,                    \label{eq:HV-RL} \\
L_n^vL_m^v-L_m^vL_n^v-(n-m)L^v_{n+m-1}, \quad n>m\ge 0,                \label{eq:HV-LvLv} \\
L_n^hL_m^v - L_m^vL_n^h - nL_{n+m-1}^h,\quad n,m\ge 0,                 \label{eq:HV-LhLv} \\
L_n^hL_m^h - L_m^hL_n^h,\quad n>m\ge 0.                                \label{eq:HV-LhLh}
\end{gather}
We will use an order on $O^*$ which is slightly different from $\preceq $ used 
in Section \ref{eq:Sect4}: assume $L_n^h>L_m^v$ for all $n,m\ge 0$.
Obviously, \eqref{eq:HV-Ld}--\eqref{eq:HV-LhLh} remains a GSB relative 
to this modified order.

Let us fix the following locality function on $X$:
\[
 N(v,v)=N(h,v)=2,\quad N(v,h)=1,\quad N(h,h)=0.
\]
Then relations
\eqref{eq:Mod_M:Locality} and \eqref{eq:Mod_M:RightMul} 
turn into 
\begin{gather}
L_n^vh , \ n\ge 1,
\quad 
L_m^hh , \ m\ge 0,           \label{eq:HV-Loc_vh} \\
L_n^vv , \  L_n^hv , \  n\ge 2,        
                             \label{eq:HV-Loc_hv} \\
R_0^vh-L_0^hv+\partial L_1^hv,   
\quad R_1^vh+L_1^hv,
\quad 
R^v_nh, \  n\ge 2,            \label{eq:HV-Rvh} \\
R_0^vv-L_0^vv+\partial L_1^vv,
\quad
R_1^vv+L_1^vv,
\quad 
R_n^vv, \  n\ge 0,            \label{eq:HV-Rvv} \\
R_0^hv-L_0^vh, \quad 
R_n^hv,\  n\ge 1,\quad 
R_m^hh, \  m\ge 0.            \label{eq:HV-Rhvh} 
\end{gather}
Commutation relations 
\eqref{eq:Mod_M:Comm} turn into 
\begin{equation}\label{eq:CommHV}
 L_1^vv-v,\quad L_1^hv-h,\quad L_0^hv-L_0^vh.
\end{equation}
It is straightforward to check that the set $S$ of relations 
\eqref{eq:HV-Ld}--\eqref{eq:HV-Rhvh}
form a GSB in the free $A(L;X)$-module generated by~$X$. 
For example, let us show that the composition of intersection
of 
$f=R_0^vL_0^h-L_0^hR_0^v$ and $g=L_0^hv-L_0^vh$
is trivial:
\[
\begin{aligned}
(f,g)_{R_0^vL_0^hv} & =  (R_0^vL_0^h-L_0^hR_0^v)v - R_0^v(L_0^hv-L_0^vh)
= R_0^vL_0^vh - L_0^hR_0^vv \\
 & \equiv  L_0^v R_0^v h - L_0^hR_0^vv 
   \equiv  L_0^v (L_0^hv-\partial L_1^hv) - L_0^h(L_0^vv-\partial L_1^vv)    \\
 & \equiv (L_0^vL_0^h - L_0^hL_0^v)v -\partial L_0^vL_1^hv + \partial L_0^hL_1^vv  
   \equiv \partial L_0^hL_1^vv -\partial L_0^vL_1^hv \\
 & \equiv L_0^h v - L_0^vh \equiv 0 \pmod{\Sigma\cup S, R_0^vL_0^hv}.
\end{aligned}
\]
The set of $(\Sigma\cup S)$-reduced normal words consists of 
\[
 \partial^s (L_0^v)^k a,\quad a\in X,\ s,k\ge 0.
\]
These words form a linear basis of $U(L; X,N)$. 

\begin{example}
 Let us evaluate the GSB of the universal associative envelope 
 of the Virasoro conformal algebra relative to $N(v,v)=3$.
\end{example}

Here we will use a different ordering of generators $O$. 
Denote $L_n=L^v_n$, $R^v_n = R_n$, suppose 
\[
 L_0<L_1<\partial <L_2 < \dots <R_0 <R_1 <\dots , 
\]
and compare monomials in $O^*$ first by their degree in $R_n$, 
next by deg-lex rule. 

In these settings, $A(L;X)$ is generated by $O$ relative to the 
following defining relations:
\begin{gather}
 \partial L_0 - L_0\partial , 
			      \label{eq:Vir1}  \\
 \partial L_1 - L_1\partial +L_0, 
			      \label{eq:Vir2}  \\
 L_n\partial - \partial L_0 -nL_{n-1},\quad n\ge 2, 
			      \label{eq:Vir3}  \\
 R_n\partial -\partial R_n - nR_{n-1},\quad n\ge 0,
			      \label{eq:Vir4}  \\
 R_nL_m - L_mR_n,\quad n,m\ge 0,
			      \label{eq:Vir5}  \\
 L_nL_m - L_mL_n -(n-m)L_{n+m-1},\quad n>m\ge 0.
			      \label{eq:Vir6}  
\end{gather}
Polynomials \eqref{eq:Vir1}--\eqref{eq:Vir6} form a GSB in $\Bbbk\langle O\rangle $ by 
Proposition~\ref{prop:U(L;X,N)-algebra}.

As an $A(L;X)$-module, $U(\Vir; v,3)$ is generated by $X=\{v\}$ 
with the following relations:
\begin{gather}
 L_n v ,\quad n\ge 3,
 			      \label{eq:Vir7}  \\
 R_0v-2L_0v+L_1\partial v -\partial^{(2)}L_2v,
 			      \label{eq:Vir8}  \\
 R_1v+L_1v-\partial L_2v,
 			      \label{eq:Vir9}  \\
 R_2v-L_2v,
 			      \label{eq:Vir10}  \\
 R_nv, \quad n\ge 3,
 			      \label{eq:Vir11}  \\
 \partial L_2v-2L_1v+2v.
 			      \label{eq:Vir12}  \\
\end{gather}
Relations \eqref{eq:Vir1}--\eqref{eq:Vir12} have the following compositions of intersection:
\begin{itemize}
 \item \eqref{eq:Vir5} with \eqref{eq:Vir7}, $w=R_mL_nv$, $n\ge 3$;
 \item \eqref{eq:Vir6} with \eqref{eq:Vir7}, $w=L_mL_nv$, $m>n\ge 3$;
 \item \eqref{eq:Vir3} with \eqref{eq:Vir12}, $w=L_n\partial L_2v$, $n\ge 2$;
 \item \eqref{eq:Vir4} with \eqref{eq:Vir12}, $w=R_mL_nv$, $n\ge 0$.
\end{itemize}
Calculation of these compositions gives only one new defining relation
\begin{equation}\label{eq:Vir13}
 L_2L_2v .
\end{equation}
Denote by $S$ the set of relations \eqref{eq:Vir1}--\eqref{eq:Vir13}
Relation \eqref{eq:Vir13} has a composition of intersection with \eqref{eq:Vir6}, $w=L_nL_2L_2v$, $n\ge 3$, which 
is trivial modulo $S$ and $w$. Yet another composition is obtained from \eqref{eq:Vir5} and \eqref{eq:Vir13}
relative to $w=R_nL_2L_2v$, $n\ge 0$. This composition is obviously trivial for $n\ge 2$. Let us show as an example the 
computation of that composition for $n=0$:
\[
 \begin{aligned}
R_0L_2L_2v  & \equiv   L_2L_2R_0v \equiv L_2L_2(L_0v -\partial L_1v + \partial^{(2)}L_2v) \\
& \equiv L_2L_2L_0v -L_2\partial L_2L_1v - 2L_2L_1L_1v + 2L_2\partial L_1L_2v + L_2L_0L_2v \\
& \equiv L_2L_2L_0v - 2L_1L_2L_1v -2L_2L_1L_1v  +4L_1L_1L_2v + L_2L_0L_2v \\  
& \equiv 2L_2L_0L_2v + 2L_2L_1 v - 2L_1L_2v-2L_2L_1L_1v+2L_1L_1L_2v  \\
& \equiv 4L_1L_2v + 2L_1L_2v + 2L_2v + 2L_1L_1L_2v -2L_1L_2v -2L_2L_1L_1 v \\
& \equiv 4L_1L_2v +2L_2v +2L_1L_1L_2v -2L_1L_1L_2v -4L_1L_2v -2L_2v \\
& \equiv 0 \pmod {S,R_0L_2L_2v }.
 \end{aligned}
\]
As a result, 
$S$ is a GSB of $U(\Vir; v,3)$.

\begin{remark}
Note that Proposition \ref{prop:U(L;X,N)-algebra} remains valid for conformal superalgebras provided that 
we add the appropriate parity to \eqref{eq:CommL-operators}:
\[
L^a_nL^b_m  -(-1)^{p(a)p(b)} L^b_mL_n^a  - \sum\limits_{s\ge 0}\binom{n}{s} L^{g_s^{a,b}}_{n+m-s} ,
\quad a,b\in X,\ n,m\in \mathbb Z_+,\ L_n^a>L_m^b.
\]
Therefore, in order to find a GSB for the universal envelope $U(L;X,N)$ of a 
Lie conformal superalgebra, 
we have to find a GSB of the $A(L;X)$-module generated by $X$ relative to the defining relations 
\eqref{eq:Mod_M:RightMul}, \eqref{eq:Mod_M:Locality}, 
and \eqref{eq:Mod_M:Comm} with appropriate parity:
\[
 L_n^ab - (-1)^{p(a)p(b)} R_n^ab - g_n^{a,b},\quad a,b\in X,\ n\in \mathbb Z_+.
\]
\end{remark}

\noindent
{\bf Acknowledgements.}
The work was supported by the Program of fundamental scientific researches of the Siberian Branch of 
Russian Academy of Sciences, I.1.1, project 0314-2016-0001.

\end{document}